\documentclass[leqno,twoside, 12pt]{amsart}
\usepackage{amssymb, latexsym, esint}
\usepackage{color}

\usepackage{amsthm}
\usepackage{amssymb,amsfonts}
\usepackage[all,arc]{xy}
\usepackage{enumerate}
\usepackage{mathrsfs}
\usepackage{amsmath}
\usepackage{verbatim}

\theoremstyle{plain}
\numberwithin{equation}{section} \numberwithin{figure}{section}

\newtheorem{theorem}{Theorem}[section]
\newtheorem{lemma}[theorem]{Lemma}
\newtheorem{proposition}[theorem]{Proposition}

\newtheorem{definition}[theorem]{Definition}
  \usepackage{epsfig} 
\usepackage{graphicx}  \usepackage{epstopdf}
\theoremstyle{definition}
\newtheorem{remark}[theorem]{Remark}

\numberwithin{equation}{section}

\def\diver{\mathop{\text{\normalfont div}}}
\newcommand{\R}{\mathbb{R}}
\newcommand{\N}{\mathbb{N}}  
\usepackage[bookmarksnumbered,colorlinks, linkcolor=blue, citecolor=red, pagebackref, bookmarks, breaklinks]{hyperref}
\newcommand{\cd}{\rightharpoonup}

\begin{document}

\title[The H\'enon equation in Orlicz-Sobolev spaces]{The H\'enon equation in Orlicz-Sobolev spaces}

\author{Pablo Ochoa}

\address{P. Ochoa. \newline Universidad Nacional de Cuyo, Fac. de Ingenier\'ia. CONICET. Universidad J. A. Maza\\Parque Gral. San Mart\'in 5500\\
Mendoza, Argentina.}
\email{pablo.ochoa@ingenieria.uncuyo.edu.ar}

\author{Ariel Salort}
\address{A. Salort \newline Departamento de Matem\'aticas y Ciencia de Datos, Universidad San Pablo-CEU, CEU Universities, Urbanizaci\'on Montepr\'incipe, 28660 Boadilla del Monte, Madrid, Spain. }
\email{\tt ariel.salort@ceu.es}

\parskip 3pt

\maketitle

\begin{abstract}In this paper, we consider the H\'enon problem in the setting of Orlicz-Sobolev spaces:
\begin{equation}\label{Henon eq abs}
\begin{cases}
-\Delta_g u= |x|^\alpha h( u) \quad \text{in }B\\
u>0 \quad \text{in }B\\
u= 0 \quad \text{on }\partial B\\
\end{cases}
\end{equation}where $B$ is the unit ball in $\mathbb{R}^n$,  $g=G'$, $h=H'$ are N-functions and the operator $-\Delta_g$ is the $g$-Laplacian. We show that the symmetric term $|x|^\alpha$, for $\alpha>0$, allows to have radial solutions even for supercritical $H$, generalizing results for the classical H\'enon equation. We also show that radial solutions are indeed bounded. Finally, we  state a Pohozaev's identity  in Orlicz-Sobolev spaces that we apply to get a range in $\alpha$ for which problem \eqref{Henon eq abs} has no bounded solutions.

\end{abstract}

\section{Introduction}
The classical H\'enon's equation \cite{H} with Dirichlet boundary conditions is given by
\begin{equation}\label{H eq}
\begin{cases}
-\Delta u= |x|^\alpha u^{q-1} \quad \text{in }B\\
u>0\quad \text{in }B\\
u=0 \quad \text{on }\partial B,
\end{cases}
\end{equation}where $B$ is the unit ball in $\mathbb{R}^n$, $q>1$ and $\alpha>0$. In \cite{H}, a related model, called spheric shell model, was used to investigate numerically the stability of spherical steady-state stellar systems with respect to spherical disturbances.

This equation has been the starting point of extensive research on existence and nonexistence results, symmetry of ground states, asymptotic behavior of solutions, and other related properties. The literature on the subject is now vast and it is impossible to give a comprehensive list of references; we refer to the seminal works \cite{AR,BW,  Nazarov,Ni,   NiSerrin,SerrinTang,   SWS} and the references therein.

In \cite{AR}, a similar problem to \eqref{H eq} was considered but with more general lower order terms $f(x, u)$. It was showed there that if the growth of $f$ is less than $|u|^{q-1}$, with $q<2^{*}:=2n/(n-2)$, then a solution exists. On the other hand, the Pohozaev's identity shows that no solution exist for $\alpha=0$,  $q\geq 2^{*}$, and for a star-shaped domain.  A remarkable progress in the study of \eqref{H eq} was made in the work \cite{Ni}. It was shown there that the addition of the symmetric term $|x|^{\alpha}$ allows   the existence of radial solutions even for the critical case and beyond. Indeed, it was proved that a positive and radial solution exists for \eqref{H eq} whenever the following relation holds
\begin{equation}\label{Ni cond}
1<q<2_\alpha^* := 2^{*}+\dfrac{2\alpha}{n-2} = \frac{n(\alpha+2)}{n-2}. 
\end{equation}

 The proof in \cite{Ni} makes use of the well-known Mountain Pass Theorem of Ambrosetti-Rabinowitz \cite{AR} via a variational approach. We also point out that  since the term $|x|^{\alpha}$ is increasing, the celebrated theorem of \cite{GNN} about existence of symmetric solutions does not apply.

The nonlinear version of the H\'enon problem \eqref{H eq} is obtained when replacing the Laplacian with the p-Laplacian $\Delta_p:=\diver(|\nabla u|^{p-2}\nabla u)$, $p>1$. This problem was addressed in \cite{Nazarov}, and existence of a positive and radial solution was guaranteed when $p<n$ and the exponent $q$ lies within the range
\begin{equation} \label{cond p.lap}
p<q<p_\alpha^*:= p^* + \frac{\alpha p}{n-p} = \frac{n(\alpha + p)}{n-p},
\end{equation}
being $p^*$ the critical Sobolev exponent.

Other extensions of problem \eqref{H eq} have been studied in recent years, including nonlinear operators \cite{Nazarov, DP}, different boundary conditions \cite{GS,  H13,Sh}, nonlocal operators \cite{SW,L23}, and mixed local–nonlocal operators \cite{SV}, to mention just a few. However, to the best of our knowledge, no results are available when nonstandard growth operators are involved in the equation. The purpose of this article is to address this gap.
 
We consider the following further generalization of the H\'enon  problem
\begin{equation}\label{Henon eq}
\begin{cases}
-\Delta_g u= |x|^\alpha h( u) \quad \text{in }B\\
u>0 \quad \text{in }B\\
u= 0 \quad \text{on }\partial B,\\
\end{cases}
\end{equation}where $\alpha>0$, $B$ denotes the unit ball in $\mathbb{R}^n$,  $g=G'$, $h=H'$ with $G$ and $H$  being N-functions (see Section \ref{preliminaries} for details) and the operator
$$-\Delta_g u := -\diver\left(\dfrac{g(|\nabla u|)}{|\nabla u|}\nabla u \right)$$is the $g$-Laplacian of $u$. We will assume the following growth behavior on $G$ and $H$:
\begin{equation} \label{cond.intro}
p^- G(t) \leq tg(t) \leq p^+ G(t), \qquad 
q^- H(t) \leq th(t) \leq q^+ H(t), \qquad t\geq 0, 
\end{equation}
for    $p^\pm,q^\pm \in (1, n)$.  


We now turn to a discussion of our main contributions. First, we address the generalization of condition \eqref{cond p.lap} in our framework.  Indeed, given an N-function $G$ and a parameter  $\alpha>0$,  we choose the N-function $H$ such that the following relation holds:
\begin{equation}\label{cond p<q.intro}
p^+<q^-.
\end{equation}
This ensures that the nonlinearity $h(u)$ is superlinear with respect to the g-Laplacian. Moreover, we assume that  $H\ll R$ at infinity (see Section \ref{sec.prelim}) with $R$ an N-function satisfying the condition
\begin{equation}\label{cond.intro on H}
\int_0^1 r^{\alpha+n-1} R(  \|s^{1-n}\|_{L^{\widetilde G}((r,1),s^{n-1}ds)})\,dr<\infty, 
\end{equation}
where $\widetilde G$ is the complementary function to $G$, and the Luxemburg's norm is defined in Section \ref{sec.prelim}.

In the particular case of power functions, that is,  $G(t)=t^p$ and $H(t)=t^q$, $p,q>1$, we have
that \eqref{cond p<q.intro} reads as $p<q$ and condition \eqref{cond.intro on H} implies that
\begin{equation*}
\begin{split}
\int_0^1 r^{\alpha+n-1}\left(\int_r^1s^{\frac{(1-n)p}{p-1}}s^{n-1}\,ds \right)^{\frac{q(p-1)}{p}}\,dr = C\int_0^1  r^{\alpha+n-1} \left(r^{\frac{p-n}{p-1}}-1 \right)^{\frac{q(p-1)}{p}}\,dr \leq C\int_0^1  r^{\alpha+n-1 + q\frac{p-n}{p}}\,dr,
\end{split}
\end{equation*}which is finite if and only if $q<p_\alpha^*$, and hence recovering \eqref{cond p.lap}.

\bigskip

In this work, we are interested in proving existence and non-existence of weak solutions for problem \eqref{Henon eq}. To seek non negative solutions, we shall consider the following energy functional
$$
J(u)= \int_B G(|\nabla u|)\,dx -  \int_B |x|^\alpha H( u_{+})\,dx,
$$
defined on the space $X_{rad}(B)$ of suitable radial functions on the ball $B$. Under assumption \eqref{cond.intro on H}, we will see in Section \ref{preliminary results}, that $J$ is indeed well-defined. Moreover, its Frechet derivative is given by
$$
\langle J'(u),v \rangle = \int_B g(|\nabla u|)\frac{\nabla u}{|\nabla u|}\cdot \nabla v\,dx - \int_B  |x|^\alpha h( u_+) v\,dx
$$
for all $v\in X_{rad}(B)$. Hence weak solutions of \eqref{Henon eq} correspond to critical points of $J$.

Our first result concerns the existence of a weak solution to \eqref{Henon eq}.

\begin{theorem} \label{teo1}
Let $G$ and $H$ be N-functions satisfying \eqref{cond.intro}, \eqref{cond p<q.intro} and  $H\ll R$ at infinity with $R$ fulfilling \eqref{cond.intro on H}. Then there exists a positive weak solution $u\in X_{rad}(B)$ to \eqref{Henon eq}.
\end{theorem}

The proof of Theorem \ref{teo1} follows the approach introduced by \cite{Ni}. For this purpose, we require several ingredients, such as a suitable version of the Radial Lemma and the compact embeddings of the  radial Orlicz–Sobolev spaces (see Section \ref{Section.embedding}). With these tools, we can then apply an appropriate version of the mountain-pass theorem.

By following a De Giorgi's $L^2$ to $L^{\infty}$ iteration scheme, we will prove that any solution of \eqref{Henon eq} is indeed bounded. Specifically, we have the following.

\begin{theorem} \label{teo_boundedness}
Let $G$ and $H$ be N-functions satisfying \eqref{cond.intro} and \eqref{cond p<q.intro} with $q^+ < r^{-}:= \inf_{t>0} \frac{tR'(t)}{R(t)}$. Then, for any weak solution $u\in X_{rad}(B)$ to \eqref{Henon eq}, there is a constant $C>0$ such that
$$\|u\|_{L^{\infty}(B)}\leq C.$$
\end{theorem}
We point out that in the case of powers, the assumption $q^+ < r^{-}$ turns into 
$$q<p^{*}_{\alpha}, \quad \text{for any }\alpha>0.$$

The nonexistence of solutions can be obtained for instance with scaling argments in the case of homogeneous operators (see for instance \cite{RoS}). In particular, the H\'enon problem \eqref{H eq} for the p-Laplacian has no weak bounded solutions when $p<n$ and 
\begin{equation*}
q\geq p_\alpha^*.
\end{equation*}
Due to the nonhomogeneous nature of problem \eqref{Henon eq}, in our case we follow the approach of applying a suitable version of Pohožaev-Pucci-Serrin type identity \cite{PuSe} for obtaining a range for non existence of bounded solutions. More precisely, we prove first the following identity for the weak solutions of the $g$-Laplacian.

\begin{theorem}[Pohozaev's type identity] \label{Poho} Assume that $G$ is an N-function such that $G'=g\in C^{2}((0, \infty))$ and for constants $1<p^-<p^+<\infty$ it holds that
$$
(p^{-}-1)g(t) \leq g'(t)t \leq (p^+-1)g(t), \quad t>0.
$$
Let $f: \overline{B}\times \mathbb{R}\to \mathbb{R}$ be a continuous function satisfying $f(x, 0)=0$, and let $u\in L^{\infty}(B)\cap W_0^{1, G}(B)$ be a weak solution of
\begin{equation*}
\begin{cases}
-\diver\left(\dfrac{g(|\nabla u|)}{|\nabla u|} \nabla u\right)= f(x, u) \quad \text{in }B\\
u=0 \quad \text{on }\partial B.
\end{cases}
\end{equation*}
Then, the following identity holds
\begin{align}\label{Pohozaev indentity}
\begin{split}
\int_B &\left(nF(x, u)\,dx + x\cdot \nabla_x F(\cdot, u)\right)dx \\
&\qquad \qquad +\int_B \left(g(|\nabla u|)|\nabla u| - nG(|\nabla u|)\right)dx= \int_{\partial B} x\cdot \nu \left( g(|\nabla u|)|\nabla u| - G(|\nabla u|)\right) d\sigma,
\end{split}
\end{align}
with  $F(x, t)=\int_0^t f(x, s)\,ds$ and where   $\nu$ denotes the outer normal unit vector to $\partial B$.  
\end{theorem}

An application of the previous identity leads to the following non-existence result.

\begin{theorem}\label{nonexistence result} Let $G$ and $H$ be N-functions satisfying \eqref{cond.intro} and such that $q^-\geq(p^+)_\alpha^*$. Then, there is no solution $u\in L^{\infty}(B)\cap W^{1, G}_0(B)$ of problem \eqref{Henon eq}.
\end{theorem}

Beyond power functions, examples of N-functions satisfying the hypothesis of Theorems \ref{teo1}, \ref{teo_boundedness} and \ref{nonexistence result} include, for instance, the following:
\begin{itemize}
\item[(i)] Power functions exhibiting different behaviors near 0 and at infinity:
 $$G(t)=\frac{t^p}{p} + \frac{t^q}{q} \text{ with }  1<p<q<\infty
$$
In this case $p^-=p$, $p^+=q$.

\item[(ii)] Logarithmic perturbation of powers:
$$
G(t)=t^p \ln^r(e-1+t^q)\text{ with }p,q,r> 0 \text{ tales que } p+qr>1 . 
$$
In this case $p^- = p$, $p^+=p+qr$.

\item[(iii)] Log-double polynomials:
$$
G(t)=t^p (\ln (\ln (e^e-1 +t)))^s, \text{ with } p>1, s>0.
$$
In this case $p^-=p$, $p^+ = p+s$.
\end{itemize}

We highlight that, as is typical in Orlicz settings, the range conditions in Theorems~1.2 and~1.4 are not complementary, in contrast with the classical power case. Indeed, for $G(t)=t^p$ with $p>1$, the relationship between $p$ and $q$ is complementary: existence holds whenever $q < p_\alpha^*$, whereas non-existence arises for bounded solutions when $q \geq p_\alpha^*$.  However, we  point out that when the N-functions in Theorems~1.2 and~1.4 are powers, then we certainly obtain the complementary range.  

\bigskip
The article is organized as follows. In Section~\ref{preliminaries} we introduce the basic notation concerning N-functions and the functional spaces employed throughout the paper. Sections~\ref{preliminary results} and~\ref{Section.embedding} are devoted to establishing useful properties of radial functions. Our main existence result is proved in Section~\ref{Section.embedding}, while the boundedness of weak solutions is addressed in Section~\ref{sec.bound}. Finally, in Section~\ref{sec.no.exist} we prove a non-existence result.

\section{Preliminaries}\label{preliminaries}

In this section, we give some basic definitions and results related to Orlicz and Orlicz-Sobolev  spaces. 

\subsection{N-functions} \label{sec.prelim}
\begin{definition}\label{d2.1}
A function $G \colon [0, \infty) \rightarrow \mathbb{R}$ is called an N-function if it admits the representation
$$G(t)= \int _{0} ^{t} g(\tau) d\tau,$$
where the function $g$ is right-continuous for $t \geq 0$,  positive for $t >0$, non-decreasing and satisfies the conditions
$$g(0)=0, \quad g(\infty)=\lim_{t \to \infty}g(t)=\infty.$$
\end{definition}
\noindent By \cite[Chapter 1]{KR}, an N-function has also the following properties:
\begin{enumerate}
\item[(i)] $G$ is continuous, convex, increasing, even and $G(0) = 0$.
\item[(ii)] $G$ is super-linear at zero and at infinite, that is $$\lim_{x\rightarrow 0} \dfrac{G(x)}{x}=0 \quad \text{ and } \quad \lim_{x\rightarrow \infty} \dfrac{G(x)}{x}=\infty.$$
\end{enumerate}

\noindent We will assume that $g$ is extended as a odd function to the whole $\mathbb{R}$.
		
Given N-functions $G$ and $H$ we say that $G\ll H$ at infinity if
$$
\lim_{t\to\infty} \frac{G(\lambda t)}{H(t)}=0
$$
holds for any $\lambda>0$.
		
An important property for N-functions is the following:
\begin{definition}

An N-function  $G$ satisfies the $\bigtriangleup_{2}$ condition if  there exists $C > 2$ such that
\begin{equation*}
G(2x) \leq C G(x) \,\,\text{~~~for all~~} x \in \mathbb{R}_+.
\end{equation*}
\end{definition}

\noindent By \cite[Theorem 4.1, Chapter 1]{KR},  an N-function  satisfies the $\bigtriangleup_{2}$ condition if and only if there is $p^+ > 1$ such that
\begin{equation*}\label{eq p mas}
tg(t) \leq p^{+} G(t), ~~~~~\forall\, t>0.
\end{equation*}

\noindent  Associated to $G$ is  the N-function  complementary to it which is defined as follows:
\begin{equation}\label{Gcomp}
\widetilde{G} (t) := \sup \left\lbrace tw-G(w) \colon w>0 \right\rbrace .
\end{equation}

The definition of the complementary function assures that the following Young-type inequality holds
\begin{equation}\label{2.5}
st \leq G(t)+\widetilde{G} (s) \text{  for every } s,t \geq 0.
\end{equation}

By \cite[Theorem 4.3,   Chapter 1]{KR}, a necessary and sufficient condition for the N-function $\widetilde{G} $ complementary to $G$ to satisfy the $\bigtriangleup_{2}$ condition is that there is $p^{-} > 1$ such that
\begin{equation*}
p^{-} G(t)\leq 	tg(t), ~~~~~\forall\, t>0.
\end{equation*}
Therefore, from now on, we will assume that for any $t>0$
\begin{equation}\label{assumpt G}
1<p^-\le \frac{tg(t)}{G(t)} \leq p^+<\infty.
\end{equation}

\subsection{Orlicz-Lebesgue and Orlicz-Sobolev spaces}
Given an N-function $G$ 
and an open set $\Omega \subseteq \mathbb{R}^{n}$, 
we consider the spaces :
\begin{align*}
    & L^{G}(\Omega) =
\left\{ u: \Omega \to \mathbb{R} \colon 
    \int_{\Omega}G(|u(x)|)\,dx < \infty \right\},\\
    &W^{1, G}(\Omega)= 
\left\{ u \in L^{G}(\Omega)\colon 
    \int_\Omega G(|\nabla u|)\,dx < \infty \right\}.
\end{align*}
We also define the space $W_0^{1, G}(\Omega)$ as the closure of $C_0^{\infty}(\Omega)$ in $W^{1, G}(\Omega)$. We recall that under the assumption \eqref{assumpt G},  the Orlicz-Lebesgue and the Orlicz-Sobolev spaces are separable and reflexive spaces and are Banach spaces endowed with the Luxemburg's norms
\begin{align*}
\|u\|_{L^G(\Omega)} = \inf \left\{ \lambda>0\colon \int_\Omega G\left(\frac{u}{\lambda} \right)\,dx\leq 1 \right\}, \qquad \|u\|_{W^{1,G}(\Omega)}:=\|u\|_{L^G(\Omega)}+ \|\nabla  u\|_{L^G(\Omega)},
\end{align*}
respectively. In the space $W^{1,G}_0(\Omega)$ it is well-known that $\|\nabla u\|_{L^G(\Omega)}$ is an equivalent norm. 

We recall that norms and modulars can be related as follows: for any $u\in L^G(B)$ we have
\begin{equation} \label{fu.xi}
\xi_{G,-}(\|u\|_{L^G(B)}) \leq \int_B G(|u|)\,dx \leq \xi_{G,+}(\|u\|_{L^G(B)}),
\end{equation}
where $\xi_{G,-}(t)=\min\{t^{p^+}, t^{p^-}\}$ and $\xi_{G,+}=\max\{t^{p^+}, t^{p^-}\}$.

\medskip
We end the section with the following H\"{o}lder inequality (see \cite[Lemma 2.6.5]{D}).
\begin{proposition} \label{holder.mu}Let $(X, \Sigma, \mu)$ be a $\sigma$-finite, complete measure space.
If $f\in L^G(\Omega;d\mu)$, $g\in L^{\widetilde G}(\Omega;d\mu)$ then
$$
\int_\Omega |fg|\,d\mu \leq 2\|f\|_{L^G(\Omega;d\mu)} \|g\|_{L^G(\Omega;d\mu)}
$$
where we denote
$$
\|f\|_{L^G(\Omega; d\mu)}=\inf\left\{ \lambda>0\colon \int_\Omega G\left(\frac{u}{\lambda}\right)\,d\mu\leq 1\right\}.
$$
\end{proposition}

\section{ Notion of solutions and preliminary results on radial functions}\label{preliminary results}

We introduce now the definition of weak solutions to problem \eqref{Henon eq}.

\begin{definition}\label{defi wk sol} We say that $u\in W_0^{1, G}(B)$ is a weak solution of \eqref{Henon eq} if  $u>0$ a.e. in $B$ and
$$
\int_B \dfrac{g(|\nabla u|)}{|\nabla u|}\nabla u \cdot \nabla v\,dx = \int_B|x|^\alpha h( u_{+}) v\,dx,
$$
for every $v\in W_{0}^{1, G}(B)$. 
\end{definition}

 In what follows, we consider the following function space of radially symmetric Orlicz-Sobolev functions:
$$X_{rad}(B):=\left\lbrace u \in W_0^{1, G}(B), u \text{ is radial} \right\rbrace.$$

The associated energy functional to problem \eqref{Henon eq} is

$$
J(u)= \int_B G(|\nabla u|)\,dx -  \int_B |x|^\alpha H( u_{+})\,dx,
$$
defined on the space $X_{rad}(B)$. To prove that $J$ is well-defined in $X_{rad}(B)$, we will need a  decay estimate for radial functions. This is the content of the Strauss-type lemma for Orlicz-Sobolev spaces stated in the next Proposition \ref{strauss.nuevo}.

\begin{proposition} \label{strauss.nuevo}
Let $u\in X_{rad}(B)$. Then, for any $r\in (0,1)$ it holds that
$$
|u(x)| \leq C
\| \nabla u\|_{L^{G}(B)} \|s^{1-n}\|_{L^{\widetilde G}((r,1),s^{n-1}ds)},
$$
where $r=|x|$ and $C>0$ is a constant depending only of $n$.
\end{proposition}

\begin{proof}
Since $u\in X_{rad}(B)$ we may, with some abuse of notation, write $u(x)=u(|x|)=u(r)$. Moreover, $u(1)=0$.    Observe that $u(r)=-\int_r^1 u'(s)\,ds$. Consider the measure $d\mu = s^{n-1}ds$. Then, Proposition \ref{holder.mu} yields
\begin{align*}
|u(r)|&\leq \int_r^1 |u'(s)|\,ds  = \int_r^1 |u'(s)| s^{n-1} s^{1-n}\,ds=\int_r^1 |u'(s)|  s^{1-n}\,d\mu\\
&\leq 2 \| u'\|_{L^G((0,1); d\mu)} \|s^{1-n}\|_{L^{\widetilde G}((r,1),d\mu)}.
\end{align*}
Let us compute the first norm. 
Using polar coordinates, we get that
$$
\omega_n\int_0^1 G\left(\frac{|u'(s)|}{\lambda}\right)s^{n-1}\,ds = \int_0^1 \int_{S^{n-1}} G\left(\frac{|u'(s)|}{\lambda}\right)s^{n-1}\, d\sigma ds = \int_B G\left(\frac{|\nabla u(x)|}{\lambda}\right)\,dx
$$
from where it is obtained that
$$
\int_0^1 G\left(\frac{|u'(s)|}{\lambda}\right)s^{n-1}\,ds = 1 \iff \int_B \frac{1}{\omega_n}G\left(\frac{|\nabla u(x)|}{\lambda}\right)\,dx =1.
$$
This gives the following relation
$$
 \| u'\|_{L^G((0,1); d\mu)} = \| \nabla u\|_{L^{\frac{1}{\omega_n}G}(B)} \leq  c_n \| \nabla u\|_{L^{G}(B)}
$$
and the result follows.
\end{proof}

\begin{remark}Observe that when $G(t)=t^p$, $1<p<n$ we get that
\begin{align*}
\|s^{1-n}\|_{L^{\widetilde G}((r,1),d\mu)} 
&= \left( \int_r^1 s^{(1-n)p'} s^{n-1}\,ds\right)^{\frac{1}{p'}}
= \left( \int_r^1 s^\frac{-n+1}{p-1}\,ds\right)^{\frac{p-1}{p}}\\
&= \left[  \frac{p-1}{n-p}\left(r^\frac{p-n}{p-1}-1\right)\right]^{\frac{p-1}{p}}
\leq C\left(  r^\frac{p-n}{p-1}\right)^{\frac{p-1}{p}}= Cr^{1-\frac{n}{p}}.
\end{align*}
Therefore, Proposition \ref{strauss.nuevo} recovers the well-known Strauss' Lemma for radial functions on $W^{1,p}_0(B)$:
$$
|u(x)|\leq C |x|^{1-\frac{n}{p}}\|\nabla u\|_{L^p(B)}.
$$ See \cite[Lemma II.1 and Remark II.3]{Lions}.
\end{remark}

With the Strauss' Lemma at hand, we may  guarantee  that the functiona $J$ is well-defined.

\begin{proposition}\label{condition}
Assume that \eqref{cond.intro on H} holds. Then, $J(u)<\infty$ for any $u\in X_{rad}(B)$.
\end{proposition}
\begin{proof}

For  $u\in X_{rad}(B)$,  from Proposition \ref{strauss.nuevo} and the fact the $H$ is increasing we get that
\begin{align} \label{cota}
\begin{split}
\int_B |x|^\alpha H( u )\,dx  &\ =C\int_0^1 r^{\alpha+n-1}H(u(r))\,dr \\& \leq  C
\int_0^1 r^{\alpha+n-1} H(  \|s^{1-n}\|_{L^{\widetilde G}((r,1),s^{n-1}ds)} \|\nabla u\|_{L^G(B)})\,dr\\
&\leq C
\xi_{H,+}(\|\nabla u\|_{L^G(B)})
\int_0^1 r^{\alpha+n-1} H(  \|s^{1-n}\|_{L^{\widetilde G}((r,1),s^{n-1}ds)})\,dr <\infty,
\end{split}
\end{align}
in view of assumption \eqref{cond.intro on H}.
\end{proof}

\begin{remark}
We point out that the following Strauss's Lemma is given in \cite{AFS}. Assuming that $g$ is strictly increasing and that $G$ satisfies the $\Delta_2$-condition, then, there is a constant $C>0$ such that for every $u\in X_{rad}(B)$ there holds
\begin{equation*}
|u(x)|\leq G^{-1}\left(C|x|^{1-n}\int_{\mathbb{R}^n}G(|u|)+G(|\nabla u|)\,dx\right),
\end{equation*}
where $G^{-1}$ is the inverse function of $G$ and $u$ is extended by $0$ outside $B$. Hence,  by Poncar\'e inequality (see for instance \cite[Lemma 2.4]{FO}), we may write the previous expression as
\begin{equation*} 
|u(x)|\leq G^{-1}\left(C|x|^{1-n}\int_{B}G(|\nabla u|)\,dx\right).
\end{equation*}

Applying the above Strauss' Lemma to 
$v=u/\|\nabla u\|_{L^G(B)}$, since $G^{-1}$  is  increasing, we obtain that
\begin{equation*}
\begin{split}
\dfrac{|u(x)|}{\|\nabla u\|_{L^{G}(B)}} \leq G^{-1}\left(C|x|^{1-n}\int_{B}G\left(\dfrac{|\nabla u|}{\|\nabla u\|_{L^G(B)}}\right)\,dx\right)= G^{-1}\left(C|x|^{1-n} \right),
\end{split}
\end{equation*}
which implies
\begin{equation*}\label{radial}
|u(x)|\leq G^{-1}\left(C|x|^{1-n} \right)\|\nabla u\|_{L^G(B)}.
\end{equation*}
However, in contrast to Proposition \ref{strauss.nuevo}, in the power case the following decay estimate is obtained
$$
u(x)\leq C |x|^{\frac{1-n}{p}}\|\nabla u\|_{L^p(B)},
$$which is not optimal.
\end{remark}

\section{Embedding results for radial functions} \label{Section.embedding}

\begin{lemma} \label{lemma.1}Let $\mu$ be a non-negative Radon measure absolutely continuous with respect to the Lebesgue measure in $\mathbb{R}^n$.
 Let $\Omega\subset \R^n$ be a bounded domain and let $\{u_k\}_{k \in \N}$ be a bounded sequence in $L_{\mu}^G(\Omega)$ such that $\{u_k\}_{k\in\N}$ converges to $u$ a.e. Then $u\in L_\mu^G(\Omega)$ and $u_k\to u$ strongly in $L_{\mu}^R(\Omega)$ for any $R\ll G$.
\end{lemma}
\begin{proof}
Let us observe that $u\in L_{\mu}^G(\Omega)$. Indeed, by Fatou's Lemma, we get that
$$
\int_\Omega G(|u|)\,d\mu \leq \liminf_{k\to\infty} \int_\Omega G(|u_k|)\,d\mu \leq C_0,
$$where we have used that $u_k$ converges to $u$ $\mu$-a.e. as well.

Let us prove that $u_k\to u$ strongly in $L_{\mu}^R(\Omega)$ when $R\ll G$. This condition means that 
$\frac{R(t)}{G(t)}\to 0$ as  $t\to\infty$. Then, for any $\varepsilon>0$, there is $T=T(\varepsilon)>0$ such that
\begin{equation}\label{R and G}
R(t)\leq \frac{\varepsilon}{8C} G(t) \quad \text{ for all } t \geq T,
\end{equation} 
with $C=C_0 c$, being $c$ the constant in the  $\Delta_2$-condition. For further reference, define for each $k>0$, the set:
$$\Omega_{T, k}:=\left\lbrace x\in \Omega: |u_k(x)-u(x)|\geq T \right\rbrace.$$

 Since $u_k\to u$ $\mu$-a.e. in $\Omega$, by Egorov's Theorem, given $\delta>0$ there exists a measurable set $E_\delta \subset \Omega$ such that $\mu(\Omega\setminus E_\delta)<\delta$ and $u_k\to u$ uniformly in $E_\delta$. Due to the uniform convergence, given $\delta>0$, there is $K>0$ such that $|u_k(x)-u(x)|\leq \delta$ for all $x\in E_\delta$ and for all $k>K$. Since $R$ is continuous and increasing, $R(|u_k-u|)\leq R(\delta)$ in $E_{\delta}$. We choose $\delta$ such that 
\begin{equation}\label{cond delta}
\delta<\min\left\lbrace T,\dfrac{ \varepsilon}{4R(T)}\right\rbrace, \quad R(\delta)\leq \frac{\varepsilon}{2\mu(\Omega)}
\end{equation}and take the corresponding $K=K(\delta)$. Then, for $k> K$, we get from \eqref{cond delta} that
$$
\int_{E_{\delta}} R(|u_k-u|)\,d\mu \leq R(\delta)\mu(\Omega)\leq \frac{\varepsilon}{2}.
$$

We next consider the set $\Omega\setminus E_{\delta}$.  Observe that  for $k> K$, we have 
$$
\Omega\setminus E_{\delta} = \left\lbrace x\in \Omega: |u_k(x)-u(x)|\geq T, \text{ for some }k\right\rbrace \cup \left\lbrace x\in \Omega: \delta \leq |u_k(x)-u(x)|< T, \text{ for some }k\right\rbrace,
$$
and moreover $\mu(\Omega\setminus E_{\delta})<\delta$. Observe that since $\delta<T$,
$$\Omega\setminus E_{\delta}= \Omega_{T, k} \cup \Omega \setminus (\Omega_{T, k}\cap E_{\delta}).$$  Then, for a fixed $k>K$ and recalling \eqref{R and G} and \eqref{cond delta}, we have
\begin{equation}
\begin{split}
\int_{\Omega\setminus E_\delta} R(|u_k-u|)\,d\mu& =   \int_{\Omega_{T, k}} R(|u_k-u|)\,d\mu + \int_{\Omega \setminus (\Omega_{T, k}\cap E_\delta)} R(|u_k-u|)\,d\mu\\ & \leq 
 \frac{\varepsilon}{8C}\int_{\Omega_{T, k}} G(|u_k-u|)\,d\mu + \int_{\Omega \setminus (\Omega_{T, k}\cap E_\delta)} R(T)\,d\mu\\ & \leq  \frac{c\varepsilon}{8C}\left( \int_{\Omega} G(|u_k|)\,d\mu  + \int_\Omega G(|u|)\,d\mu\right)+R(T)\delta
 \leq \frac{\varepsilon}{2}.
\end{split}
\end{equation}

\medskip

Therefore, for all $k>K$,
$$
\int_\Omega R(|u_k-u|)\,d\mu = \int_{\Omega\setminus E_\delta} R(|u_k-u|)\,d\mu + \int_{E_\delta} R(|u_k-u|)\,d\mu \leq \varepsilon
$$
which gives that $u_k\to u$ strongly in $L_{\mu}^R(\Omega)$.
\end{proof}

\begin{proposition} \label{prop.1}
Let $G$ be an N-function and $\alpha>0$. Let $R$ be an N-function such that \eqref{cond.intro on H} holds. Then, if $\mu= |x|^\alpha\,dx$, we have that
\begin{itemize}
\item[(i)]the inclusion $X_{rad}(B)\subset L_{\mu}^R(B)$ is continuous;
\item[(ii)] the inclusion $X_{rad}(B)\subset L_{\mu}^H(B)$ is compact for any $H$ such that $H\ll R$ at infinity.
\end{itemize} 
\end{proposition}

\begin{proof}
First, we prove that the inclusion $X_{rad}(B)\subset L_{\mu}^R(B)$ is continuous. Let  $\left\lbrace u_k \right\rbrace_{k\in\N}$ be a sequence converging to some $u$ in $X_{rad}(B)$.  By  \eqref{cota} and assumption  \eqref{cond.intro on H}, we get
\begin{equation}
\begin{split}
\int_B R(u_k-u)\,d\mu & \leq \int_0^1 r^{\alpha+n-1} R\left(C\|\nabla u_k-\nabla u\|_{L^G(B)}\|s^{1-n}\|_{L^{\tilde{G}}((r, 1), s^{n-1}ds)}\right)\,dr\\& \leq C \xi_{R,+}(\|\nabla u_k-\nabla u\|_{L^G(B)})\int_0^1 r^{\alpha+n-1} R\left(\|s^{1-n}\|_{L^{\tilde{G}}((r, 1), s^{n-1}ds)}\right)\,dr \\ &  = C \xi_{R,+}(\|\nabla u_k-\nabla u\|_{L^G(B)}) \to 0
\end{split}
\end{equation}as $k\to \infty$. This proves item $(i)$.


To prove $(ii),$ let $\{u_k\}_{k\in\N}$ be a bounded sequence in $X_{rad}(B)$. Observe that since $H\ll R$ at infinity, $H$ also satisfies \eqref{cond.intro on H} and hence the inclusion is continuous. Next, By compactness of the embedding $W^{1, G}_{0}(\Omega)$ into $L^G(\Omega)$ for any bounded and smooth set $\Omega$, we get up to a subsequence  that $u_k\to u$  a.e. in $B$. Moreover, by  Proposition \ref{strauss.nuevo}, the sequence $\{ u_k\}_{k\in\N}$ is bounded in $L_{\mu}^R(B)$ since $R$ satisfies \eqref{condition} by assumption. Hence, by Lemma \ref{lemma.1},
$$
 u_k \to  u \text{ strongly in } L_{\mu}^H(B)
$$
for any $H\ll R$, and the proposition follows.
\end{proof}

\bigskip

\section{Existence of solutions: proof of Theorem \ref{teo1} }  \label{sec.exist}
 
\subsection{The Palais-Smale condition}

Let us start proving that under suitable assumptions, $J$ satisfies the Palais-Smale condition (see Definition \ref{def1}).

\begin{lemma} \label{lema.1}
Assume \eqref{cond.intro} and \eqref{cond p<q.intro}.  If $H\ll R$ at infinity, with $R$ satisfying \eqref{cond.intro on H}, then the functional  $J$ satisfies the Palais-Smale condition.
\end{lemma}

\begin{proof}
Let us first prove that every Palais-Smale sequence $\{u_k\}_{k\in\N}\subset X_{rad}(B)$ for $J$ is bounded. Since $J'(u_k)\to 0$ and using \eqref{cota}
\begin{align}\label{cota1}
\begin{split}
|\langle J'(u_k),u_k\rangle| &=
\left| \int_B g(|\nabla u_k|)\frac{\nabla u_k}{|\nabla u_k|}\cdot \nabla u_k\,dx - \int_B h((u_k)_+) |x|^\alpha u_k\,dx \right|\\ 
&=
\left| \int_B g(|\nabla u_k|)|\nabla u_k|\,dx - \int_B h( (u_k)_+)|x|^\alpha u_k \,dx \right|\\
&\leq 
\|J'(u_k)\|_{W^{1,G}_0(B)'} \|u_k\|_{W^{1,G}_0(B)}\\
&\leq \|\nabla u_k\|_{L^G(B)}
\end{split}
\end{align}
for $k$ large enough. In particular we get that, for $k$ sufficiently large,
\begin{align*}
q^- \int_B |x|^{\alpha} H( (u_k)_+) \,dx &\leq \|\nabla u_k\|_{L^G(B)} + p^+ \int_B G(|\nabla u_k|)\,dx
\end{align*}
Condition $|J(u_k)|\leq C$ means that
\begin{equation} \label{cota2}
\left| \int_B G(|\nabla u_k|)\,dx - \int_B |x|^\alpha H((u_k)_{+}) \,dx\right| \leq C.
\end{equation}
Then, from  \eqref{cota1} and \eqref{cota2} we get that
\begin{align*}
\int_B G(|\nabla u_k|)\,dx  &\leq C +  \int_B |x|^\alpha H( (u_k)_{+}))\,dx\leq C+ \frac{1}{q^-} \|\nabla u_k\|_{L^G(B)}  +\frac{p^+}{q^-}  \int_B G(|\nabla u_k|)|\,dx,
\end{align*}
from where, since $p^+<q^-$, we obtain, using \eqref{fu.xi} that

$$
\left(1-\frac{p^+}{q^-}\right) \xi_{G,-}(\|\nabla u_k\|_{L^G(B)}) \leq \left(1-\frac{p^+}{q^-}\right) \int_B G(|\nabla u_k|)\,dx \leq C+ \frac{1}{q^-} \|\nabla u_k\|_{L^G(B)} . 
$$

Hence, $\{u_k\}_{k\in\N}$ is bounded in $X_{rad}(B)$. Up to a subsequence we have that $u_k \cd u$ weakly in $X_{rad}(B)$.

Let us prove that
$$
\lim_{k\to\infty} \int_B  h( (u_k)_+)|x|^\alpha(u_k-u)\,dx=0.
$$
Since  $H$ is convex, 
$$
\int_B h( u_k)|x|^\alpha (u_k-u)\,dx \leq \int_B |x|^\alpha H( u_k)\,dx -\int_B |x|^\alpha H( u)\,dx \to 0 
$$
as $k\to\infty$ since $u_k \to  u$ strongly in $L_{\mu}^H(B)$ due to Proposition \ref{prop.1}. This implies that, 
\begin{align*}
|\langle -\Delta_g (u_k), u_k-u\rangle| &= |\langle J'(u_k), u_k-u\rangle + \int_B h( (u_k)_+)|x|^\alpha(u_k-u)\,dx |\\
&\leq
\|J'(u_k)\|_{W^{1,G}_0(B)'} \|u_k-u\|_{W^{1,G}_0(B)} + \int_B h( (u_k)_+)|x|^\alpha(u_k-u)\,dx \to 0
\end{align*}
as $k\to\infty$. Observe that, since $u_k\cd u$ weakly in $X_{rad}(B)$ we have that
$$
|\langle -\Delta_g (u), u_k-u\rangle| \to 0 \text{ as } k\to\infty
$$
therefore, 
\begin{align} \label{cuenta}
|\langle -\Delta_g (u_k) +\Delta_g u , u_k-u\rangle| \to 0 \text{ as } k\to\infty.
\end{align}

Finally, let us see that $u_k\to u$ strongly in $X_{rad}(B)$. By \eqref{cuenta} and  Lemma 3.1 in \cite{CSS} we obtain that
\begin{align*}
0\gets |\langle -\Delta_g (u_k) +\Delta_g (u), u_k-u\rangle| &= \left|  \int_B \left(g(|\nabla u_k|)\frac{\nabla u_k}{|\nabla u_k|} - g(|\nabla u|)\frac{\nabla u}{|\nabla u|}  \right)\cdot (\nabla u_k - \nabla u) \,dx \right|\\
&\geq C\int_B G(|\nabla u_k-\nabla u|)\,dx
\end{align*}
which concludes the proof.
\end{proof}

\subsection{Mountain-pass geometry of $J$}
Let us see now that $J$ satisfies the mountain-pass Theorem hypothesis (see Proposition \ref{mountain}).

\begin{lemma} \label{lema.2}Assume \eqref{cond.intro} and \eqref{cond p<q.intro}. If $H\ll R$ at infinity, with $R$ satisfying \eqref{cond.intro on H}, then the functional  $J$ satisfies the conditions:
\begin{itemize}
\item[(i)] $J(0)=0$ and $J(v)\leq 0$ for some $v\neq 0$ in $X_{rad}(B)$,
\item[(ii)] there exists $\alpha \in (0,\|\nabla v\|_{L^G(B)})$ and $\sigma>0$ such that $J\geq \sigma$ on
$$
S_\alpha :=\{u\in X_{rad}(B) \colon \|\nabla u\|_{L^G(B)} = \alpha\}.
$$
\end{itemize}
\end{lemma}

\begin{proof}
Let us check item (i). By definition, $J(0)=0$.  Let $u_0\in X_{rad}(B)$ be such that $u_0>0$ and $\|\nabla u_0\|=1$. Then $\int_B G(|\nabla u_0|)\,dx \leq 1$ and then, since  $q^->p^+$, 
\begin{align*}
J(tu_0)&= \int_B G(t|\nabla u_0|)\,dx - 	 \int_B |x|^\alpha H(t (u_0)_{+})\,dx\\
&\leq 
t^{p^+} - t^{q^-} \int_B |x|^\alpha H( (u_0)_{+})\,dx \leq 0
\end{align*}
for $t>1$ large enough.

Let us   prove item (ii). First,
\begin{align*} 
\begin{split}
\int_B |x|^\alpha H( u )\,dx &\leq  
C \xi_{H,+}(\|\nabla u\|_{L^G(B)})) = C \|\nabla u\|_{L^G(B)}^{q^-}
\end{split}
\end{align*}
when $\|\nabla u\|_{L^G(B)} \leq 1$. Finally, since $q^->p^+$ we get that
\begin{align*}
J(u)&= \int_B G(|\nabla u|)\,dx -  \int_B |x|^\alpha H( u_{+})\,dx\\
&\geq
\|\nabla u\|_{L^G(B)}^{p^+} - C \|\nabla u\|_{L^G(B)}^{q^-} =  \alpha^{p^+}-C \alpha^{q^-}:=\sigma >0
\end{align*}
if $\|\nabla u\|_{L^G(B)} =\alpha $ with $\alpha>0$ small enough.
\end{proof}

We are position to prove our first existence result.

\begin{proof}[Proof of Theorem \ref{teo1}]
The proof is a consequence of the Mountain Pass Theorem stated in Proposition \ref{mountain} by applying Lemmas \ref{lema.1} and \ref{lema.2}.
\end{proof}

\begin{remark}
We point out that the critical function $H=G^*$ is admissible in condition \eqref{cond.intro on H}.

Let $\widetilde G$ be the complementary function to $G$, and denote $\widetilde G'=\widetilde g$. Since \eqref{cond.intro} implies that
$$
(p^+)' \leq \frac{t \widetilde g(t)}{\widetilde G(t)}\leq (p^-)',
$$
by Lemma 2.7 in \cite{FBPLS} we have the following space inclusion
$$
L^{(p^-)'}((r,1), s^{n-1}\,ds) \subset L^{\widetilde G}((r,1), s^{n-1}\,ds).
$$
Moreover, by Remark 3.2 in \cite{FBSi} the critical function $G^*$ has the following growth behavior
$$
G^*(t) \leq C_1 + C_2 t^{(p^+)^*}
$$
where $C_1,C_2>0$ are constants, 
$(p^-)' = \frac{p^-}{p^--1}$ and  $(p^+)^* = \frac{np^+}{n-p^+}$. Then, assuming $p^-<n$ we get that
\begin{align*}
\|s^{1-n}\|_{L^{\widetilde G}(r,1),s^{n-1}ds} \leq \|s^{1-n}\|_{L^q((r,1),s^{n-1}ds)} =\left(\int_r^1 s^{\frac{p^--n}{p^--1}-1} \,ds\right)^\frac{p^--1}{p^-} \leq C r^{1-\frac{n}{p^-}}.
\end{align*}
Hence, since the N-function $G^*$ is increasing
\begin{align*}
\int_0^1 &r^{\alpha+n-1} G^*( \|s^{1-n}\|_{L^{\widetilde G}(r,1),s^{n-1}ds} ) \,dr \leq \int_0^1 r^{\alpha+n-1}\left(C_1 + C_2 \|s^{1-n}\|_{L^{\widetilde G}(r,1),s^{n-1}ds}^{(p^+)^*}\right)\,dr\\
&\leq
C_1\int_0^1 r^{\alpha+n-1}\,dr + C_2 \int_0^1 r^{\alpha+n-1+(1-\frac{n}{p^-})(p^+)^*}\,dr 
\end{align*}
which is finite when
$$
\alpha+n+(1-\frac{n}{p^-})(p^+)^*>0 \quad \text{ that is, when }  \alpha   >n\left( \frac{n-p^-}{n-p^+} \frac{p^+}{p^-}-1  \right) .
$$

\end{remark}

\section{Boundedness of solutions} \label{sec.bound}
 
In this section, we prove that any radial solution 
u of problem \eqref{Henon eq} is bounded. Our argument begins with the following technical lemma concerning the composition of M-functions.

\begin{lemma}\label{F function}Let $R$ and $H$ be N-functions satisfying \eqref{cond.intro}, where $r^{\pm}>1$ are the corresponding coefficients for $R$. Assume moreover that $q^{+}< r^{-}$. Then, the composition  $F:=R\circ H^{-1}$ is an N-function. 
\end{lemma}
\begin{proof}
The proof follows exactly the lines of \cite[Lemma 8]{FSV} with $R$ replacing $G^{*}$ in that argument.
\end{proof}

\begin{proof}[Proof of Theorem \ref{teo_boundedness}] In what follows, we will apply a De Giorgi's iteration scheme to control the level sets of a solution $u\in X_{rad}(B)$ to problem \eqref{Henon eq}.

For a positive integer $k$, define
$$w_{k}:= (u-(1-2^{-k}))_{+}.$$Then, as in \cite{FP}, the following holds
\begin{equation}\label{property w}
w_{k+1} \leq w_k \text{ in }B, \quad u(x)< (2^{k+1}-1)w_k \text{ in }\left\lbrace w_{k+1}>0\right\rbrace, \quad \text{and }\left\lbrace w_{k+1}>0\right\rbrace \subset \left\lbrace w_{k}>2^{-(k+1)}\right\rbrace. 
\end{equation} Also, $0 \leq w_k \leq |u|+1 \in L^1(B)$ and $w_k(x)\to (u(x)-1)_{+}$ a.e. in $B$, so by dominated convergence theorem,
\begin{equation}\label{convergence to u}
\lim_{k\to \infty}\int_B |x|^{\alpha}H(w_k)\,dx = \int_B |x|^\alpha H((u-1)_{+})\,dx.
\end{equation}
Now, since $u$ is a solution of \eqref{Henon eq} and appealing to \eqref{cond.intro}, we obtain
\begin{equation}\label{ineq b 1}
\begin{split}
\int_B G(|\nabla w_{k+1}|)\,dx & \leq \dfrac{1}{p^{-}}\int_B \dfrac{g(|\nabla w_{k+1}|)}{|\nabla w_{k+1}|}\nabla w_{k+1}\cdot \nabla w_{k+1}\,dx \\ & =  \dfrac{1}{p^{-}} \int_B \dfrac{g(|\nabla u|)}{|\nabla u|}\nabla u\cdot \nabla w_{k+1}\,dx = \dfrac{1}{p^{-}}\int_B |x|^{\alpha}h(u)w_{k+1}\,dx,
\end{split}
\end{equation}where we have tested with $w_{k+1}$. By \eqref{property w},
$$\int_B |x|^{\alpha}h(u)w_{k+1}\,dx \leq \int_B |x|^{\alpha}h\left((2^{k+1}-1)w_k\right)w_{k}\,dx \leq q^{+}(2^{k+1}-1)^{q^{+}-1}\int_B |x|^\alpha H(w_k)\,dx.$$
Thus, from \eqref{ineq b 1}, we obtain
\begin{equation}\label{ineq 100}
\int_B G(|\nabla w_{k+1}|)\,dx \leq \dfrac{q^{+}}{p^{-}}(2^{k+1}-1)^{q^{+}-1}\int_B |x|^\alpha H(w_k)\,dx = C^{(1)}_{k+1}\int_B |x|^\alpha H(w_k)\,dx.
\end{equation}

On the other hand, by Lemma \ref{F function} and H\"{o}lder's inequality, and letting $d\mu =|x|^{\alpha}\,dx$, we get
\begin{equation}\label{bound H and chi}
\int_B |x|^\alpha H(w_{k+1})\,dx \leq 2\|H(w_{k+1})\|_{L_{\mu}^F(B)}\|\chi_{w_{k+1}>0}\|_{L_\mu^{\tilde{F}}(B)},
\end{equation}where $F$ is the N-function $F=R\circ H^{-1}$. By \cite[Corollary 7]{RR}, we obtain
\begin{equation}\label{bound chi}
\begin{split}
\|\chi_{w_{k+1}>0}\|_{L_\mu^{\tilde{F}}(B)}&  = \mu(\{w_{k+1}>0\})(H\circ R^{-1})\left(\dfrac{1}{\mu(\{w_{k+1}>0\})} \right) \\ & \leq \mu(\{w_{k+1}>0\}) \psi_1 \left( \dfrac{1}{\mu(\{w_{k+1}>0\})} \right),
\end{split}
\end{equation}
where for $t\geq 0$ the function $\psi_1(t)$ is defined as $\psi_1(t):=\max\{t^\frac{q^+}{r^+}, t^\frac{q^+}{r^-}, t^\frac{q^-}{r^-}, t^\frac{q^-}{r^-}\}$.

By the assumption $q^{+}<r^{-}$, all the exponents satisfy 
$$\frac{q^{+}}{r^{+}}, \quad \frac{q^{-}}{r^{+}}, \quad \frac{q^{-}}{r^{-}}, \quad \frac{q^+}{r^-} <1.$$Hence, there exists $\delta>0$, which could be equal to $1-q^{\pm}/r^{\pm}$ according to the size of the level set $\{w_{k+1}>0\}$ such that
$$\|\chi_{w_{k+1}>0}\|_{L_\mu^{\tilde{F}}(B)} \leq \mu(\{w_{k+1}>0\})^{\delta}.$$As a result, appealing to \eqref{property w}, if we denote $C_{k+1}^{(2)}:=2^{\delta q^{+}(k+1)}$, we get that
\begin{equation}\label{ineq norm 1}
\begin{split}
\|\chi_{w_{k+1}>0}\|_{L_\mu^{\tilde{F}}(B)} & \leq \left(\int_{\{w_{k+1}>0\}}|x|^\alpha \,dx  \right)^{\delta}  = \left(\int_{\{w_{k+1}>0\}}\dfrac{H(2^{-(k+1)})}{H(2^{-(k+1)})} |x|^\alpha \,dx  \right)^{\delta} \\& \leq C_{k+1}^{(2)} \left(\int_{\{w_{k+1}>0\}}H(w_k) |x|^\alpha \,dx  \right)^{\delta}.
\end{split} 
\end{equation}

 By \cite[Lemma 10]{FSV} with $G^{*}$ replaced by $R$, the embedding in Proposition \ref{prop.1} and \eqref{ineq 100},
\begin{equation}\label{bound with max}
\begin{split}
\|H(w_{k+1})\|_{L_{\mu}^F(B)} & \leq \max\left\lbrace \|w_{k+1}\|_{L^R_\mu(B)}^{q^{+}},\|w_{k+1}\|_{L^R_\mu(B)}^{q^{-}}\right\rbrace \\& \leq C\max\left\lbrace \|\nabla w_{k+1}\|_{L^G(B)}^{q^{+}},\|\nabla w_{k+1}\|_{L^G(B)}^{q^{-}}\right\rbrace \\ & \leq C \psi_2\left(\int_B G(|\nabla w_{k+1}|)\,dx \right)\\
& \leq C^{(3)}_{k+1} \psi_2\left(\int_B |x|^\alpha H(w_k)\,dx\right),
\end{split}
\end{equation}where $C^{(3)}_{k+1}\geq C^{i}_{k+1}$ for $i=1, 2$, and  $\psi_2(t):=\max\{t^\frac{q^+}{p^+}, t^\frac{q^+}{p^-}, t^\frac{q^-}{p^-}, t^\frac{q^-}{p^-}\}$.  Next, suppose that
\begin{equation}
\int_B |x|^\alpha H(w_0)\,dx < 1.
\end{equation}Then, since $\left\lbrace w_{k} \right\rbrace_k $ is decreasing, it follows that
\begin{equation}\label{extra assump}
\int_B |x|^\alpha H(w_k)\,dx<1,
\end{equation}for all $k$. Hence, combining  \eqref{bound H and chi} with the bounds \eqref{ineq norm 1} and \eqref{bound with max}, using \eqref{extra assump} with the assumption $p^{+}<q^{-}$, we derive
 \begin{equation}
 \int_B H(w_{k+1})|x|^{\alpha}\,dx \leq C_{k+1} \left(\int_B H(w_{k+1})|x|^{\alpha}\,dx\right)\left( \int_B H(w_{k+1})|x|^{\alpha}\,dx\right)^{\delta+\gamma},
 \end{equation}with
 $$\gamma = \min\left\lbrace \dfrac{q^{+}}{p^{+}},\dfrac{q^{+}}{p^{.}}, \dfrac{q^{-}}{p^{+}}, \dfrac{q^{-}}{p^{-}}\right\rbrace -1>0,$$and where $C_{k+1}\geq C^{(3)}_{k+1}$. Therefore, by the numerical lemma \cite[Lemma 13]{FSV}, we obtain that there is $\varepsilon\in (0, 1)$ such that if
 \begin{equation}\label{assumption with eps}
 \int_B H(w_0)|x|^\alpha\,dx<\varepsilon
 \end{equation}then
 $$\lim_{k\to \infty} \int_B H(w_{k})|x|^{\alpha}\,dx =0.$$Observe that 
 $$\int_B H(w_0)|x|^\alpha\,dx=\int_B H(u_{+})|x|^\alpha\,dx$$and moreover \eqref{convergence to u} implies  $\|u\|_{L^{\infty}(B)}\leq 1$. Hence, the result is valid under the assumption
 $$\int_B H(u_{+})|x|^\alpha\,dx<\varepsilon,$$which in particular implies \eqref{extra assump}. If now
 $\xi=\int_B H(u)|x|^{\alpha}\,dx>0$ is arbitrary, then we choose  $C>1$ large enough so that
 \begin{equation}
 \int_B H\left(\dfrac{u}{C}\right)|x|^{\alpha}\,dx \leq C^{-q^{-}}\int_B H(u)|x|^{\alpha}\,dx<\varepsilon.
 \end{equation}Moreover, $u/C$ also satisfies \eqref{ineq b 1} with an extra constant at the end. Hence, from the above argument 
 $$\|u/C\|_{L^{\infty}(B)}\leq 1 \text{ or }\|u\|_{L^{\infty}(B)}\leq C.$$This ends the proof of the theorem.
\end{proof}

\section{A non-existence result for the H\'enon problem } \label{sec.no.exist}

We start this section by proving the Pohozaev's type  identity given in Proposition \ref{Poho}.  
\begin{proof}[Proof of Theorem \ref{Poho}]
First, by \cite[Theorem 1.7]{L} and subsequent remarks, we conclude that $u\in C^{1, \alpha}(\overline{B})$, for some $\alpha\in (0, 1)$. Hence, $f(x, u(x))$ is continuous and vanishes on $\partial B$. We now regularized the problem.

For simplicity in the notation, we will denote $G(t):=\int_0^t a(s)s\,ds$.  Moreover, from the relation
$$
\dfrac{a'(t)t}{a(t)}= \dfrac{tg(t)+t^2g'(t)}{tg(t)}=1+\dfrac{tg'(t)}{g(t)}
$$
we have that
$$
-1<p^{-}-2\leq \dfrac{a'(t)t}{a(t)} \leq p^{+}-2<\infty.
$$
Let $\{f_\varepsilon\}_{\varepsilon>0} \subset C_0^{2}(B)$ be a sequence converging uniformly to $f(x, u(x))$ in $\overline{B}$. Consider the problem 
 \begin{equation}\label{auxiliary problem} 
\begin{cases}
-\diver(a(\sqrt{\varepsilon+|\nabla u|^2})\nabla u)= f_\varepsilon \quad \text{in }B\\
u=0 \quad \text{on }\partial B.
\end{cases}
\end{equation}As in the beginning of the proof of Theorem 1.1 in \cite{CM}, the solution $u_\varepsilon \in C^3(\overline{B})$ and moreover, there is a constant $C=C(p^{-}, p^{+}, n, B, \|f\|_{\infty})>0$ such that
$$\|u_\varepsilon\|_{C^{1, \alpha}(\overline{B})}\leq C.$$Hence, $u_\varepsilon \to u$ in $C^{1, \alpha'}(\overline{B})$ for $\alpha'<\alpha$. 
Let $P=a(\sqrt{\varepsilon+|\nabla u_\varepsilon|^2})(\nabla u\cdot x)\nabla u_\varepsilon$. The divergence theorem implies that
\begin{equation}\label{divergence theorem}
\int_B \text{div } P \,dx=\int_{\partial B} (P\cdot \nu)\,d\sigma.
\end{equation}Firstly, observe that
\begin{equation}\label{int2}
\int_{\partial B} (P\cdot \nu)\,d\sigma =\int_{\partial B}a(\sqrt{\varepsilon+|\nabla u_\varepsilon|^2})(\nabla u_{\varepsilon}\cdot x)(\nabla u_\varepsilon \cdot \nu)\,d\sigma.
\end{equation}On the other hand,
\begin{equation}
\begin{split}
\text{div } P&= (\nabla u_\varepsilon \cdot x)\diver\left( a(\sqrt{\varepsilon+|\nabla u_\varepsilon|^2})\nabla u_{\varepsilon}\right) + a(\sqrt{\varepsilon+|\nabla u_\varepsilon|^2})\nabla u_{\varepsilon} \cdot \nabla (\nabla u_\varepsilon \cdot x)\\& =(\nabla u_\varepsilon \cdot x)\diver\left( a(\sqrt{\varepsilon+|\nabla u_\varepsilon|^2})\nabla u_{\varepsilon}\right) + a(\sqrt{\varepsilon+|\nabla u_\varepsilon|^2})\left(|\nabla u_\varepsilon|^2+ \dfrac{1}{2}x\cdot \nabla(|\nabla u_\varepsilon|^2) \right).
\end{split}
\end{equation}Hence, 
\begin{equation}\label{int0}
\begin{split}
\int_B \text{div } P\,dx & = \int_B (\nabla u_\varepsilon \cdot x)\diver\left( a(\sqrt{\varepsilon+|\nabla u_\varepsilon|^2})\nabla u_{\varepsilon}\right) \,dx + \int_B  a(\sqrt{\varepsilon+|\nabla u_\varepsilon|^2})|\nabla u_\varepsilon|^2\,dx\\ & \quad + \dfrac{1}{2}\int_B  a(\sqrt{\varepsilon+|\nabla u_\varepsilon|^2}) x\cdot \nabla (|\nabla u_\varepsilon|^2)\,dx \\& = -\int_B (\nabla u_\varepsilon \cdot x)f_\varepsilon \,dx + \int_B u_\varepsilon f_\varepsilon\,dx + \dfrac{1}{2}\int_B x\cdot \nabla \left(\int_0^{|\nabla u_\varepsilon|^2}a(\sqrt{\varepsilon + s})\,ds \right)\,dx,
\end{split}
\end{equation}where we have used that $u_\varepsilon$ is a solution of problem \eqref{auxiliary problem}. Now, by integration by parts,
\begin{equation}\label{in1}
\begin{split}
\int_B x\cdot \nabla \left(\int_0^{|\nabla u_\varepsilon|^2}a(\sqrt{\varepsilon + s})\,ds \right)\,dx & = -n\int_B\left(\int_0^{|\nabla u_\varepsilon|^2}a(\sqrt{\varepsilon + s})\,ds \right)\,dx \\& +\int_{\partial B}(x\cdot \nu) \left(\int_0^{|\nabla u_\varepsilon|^2}a(\sqrt{\varepsilon + s})\,ds \right)\,d\sigma.
\end{split}
\end{equation}As $u_\varepsilon=0$ on $\partial B$, we have that $\nabla u_\varepsilon=(\nabla u_\varepsilon \cdot \nu)\nu$ on $\partial B$. 

As a result of \eqref{int0}, \eqref{in1} and \eqref{int2}, we get
\begin{equation}\label{first P}
\begin{split}
&-\int_B (\nabla u_\varepsilon \cdot x)f_\varepsilon \,dx + \int_B u_\varepsilon f_\varepsilon\,dx  - \dfrac{n}{2}\int_B\left(\int_0^{|\nabla u_\varepsilon|^2}a(\sqrt{\varepsilon + s})\,ds \right)dx \\ & \qquad  = -\dfrac{1}{2}\int_{\partial B}(x\cdot \nu) \left(\int_0^{|\nabla u_\varepsilon|^2}a(\sqrt{\varepsilon + s})\,ds \right)\,d\sigma +\int_{\partial B}a(\sqrt{\varepsilon+|\nabla u_\varepsilon|^2})(\nabla u_{\varepsilon}\cdot x)(\nabla u_\varepsilon \cdot \nu)d\sigma \\ & \qquad = -\dfrac{1}{2}\int_{\partial B}(x\cdot \nu) \left(\int_0^{|\nabla u_\varepsilon|^2}a(\sqrt{\varepsilon + s})\,ds \right)\,d\sigma +\int_{\partial B}a(\sqrt{\varepsilon+|\nabla u_\varepsilon|^2})|\nabla u_\varepsilon|^2  ( x\cdot \nu)\,d\sigma \\ & \qquad =\int_{\partial B}(x\cdot \nu) \left[a(\sqrt{\varepsilon+|\nabla u_\varepsilon|^2})|\nabla u_\varepsilon|^2-\dfrac{1}{2} \left(\int_0^{|\nabla u_\varepsilon|^2}a(\sqrt{\varepsilon + s})\,ds \right)\right]\,d\sigma
\end{split}
\end{equation}To treat the integral
$$\int_0^{|\nabla u_\varepsilon|^2}a(\sqrt{\varepsilon + s})\,ds$$for each $x$ fixed, we perform the change of variables
$t=\sqrt{\varepsilon+s}$, which gives $2t\,dt= ds$ and we obtain
$$ \int_0^{|\nabla u_\varepsilon|^2}a(\sqrt{\varepsilon + s})\,ds= 2\int_{\sqrt{\varepsilon}}^{\sqrt{\varepsilon +|\nabla u_\varepsilon|^2}}a(t)t\,dt= 2 \left( G(\sqrt{\varepsilon +|\nabla u_\varepsilon|^2})-G(\sqrt{\varepsilon})\right).
$$
Hence, we have from \eqref{first P},
\begin{equation}\label{second P}
\begin{split}
&-\int_B (\nabla u_\varepsilon \cdot x)f_\varepsilon \,dx + \int_B u_\varepsilon f_\varepsilon\,dx-n\int_B\left( G(\sqrt{\varepsilon +|\nabla u_\varepsilon|^2})-G(\sqrt{\varepsilon})\right)\,dx \\ & \quad= \int_{\partial B}(x\cdot \nu) \left[a(\sqrt{\varepsilon+|\nabla u_\varepsilon|^2})|\nabla u_\varepsilon|^2-\left( G(\sqrt{\varepsilon +|\nabla u_\varepsilon|^2})-G(\sqrt{\varepsilon})\right)\right] d\sigma.
\end{split}
\end{equation}
As $\varepsilon\to 0^{+}$, we obtain from \eqref{second P},
\begin{align}\label{third P}
\begin{split}
-\int_B (\nabla u \cdot x)f(x, u) \,dx &+ \int_B u f(x, u)\,dx  - n\int_B G(|\nabla u|)\,dx=\\
&=\int_{\partial B}(x\cdot \nu) \left[a(|\nabla u|)|\nabla u|^2- G(|\nabla u|)\right] d\sigma.
\end{split}
\end{align}
Finally, observe that from \eqref{third P} we obtain \eqref{Pohozaev indentity} since
$$\int_B u f(x, u)\,dx = \int_B a(|\nabla u|)|\nabla u|^2\,dx$$and by integration by parts
\begin{align*}
\int_B (\nabla u \cdot x)f(x, u) \,dx&=\int_B x\cdot \nabla_x F(x, u)\,dx - \int_B x\cdot \nabla_x F(\cdot, u)\,dx\\ 
&= -n\int_B F(x, u)\,dx - \int_B x\cdot \nabla_x F(\cdot, u)\,dx,
\end{align*}
where
$\nabla_x F(\cdot, u)$ denotes the gradient of $F$ with respect to the first entry.
\end{proof}

Now, we give the proof of Theorem \ref{nonexistence result}. 

\begin{proof}[Proof of Theorem \ref{nonexistence result}] Let $f(x, u)=|x|^\alpha h(u)$ and then 
$$F(x, u)= |x|^\alpha H(u), \quad \nabla_xF(\cdot, u)= \alpha |x|^\alpha H(u).$$Moreover, observe that for each $x\in \partial B$, $\nu=x$. Moreover, by \eqref{assumpt G},
$$a(|\nabla u|)|\nabla u|^2-G(|\nabla u|)= g(|\nabla u|)|\nabla u|-G(|\nabla u|) \geq \left(1-\dfrac{1}{p^{+}} \right) g(|\nabla u|)|\nabla u|\geq 0$$which is positive in a subset of $B$ of positive measure. Hence, since $\nu = x$ at $x\in \partial B$, we get
$$\int_{\partial B}(x\cdot \nu) \left[a(|\nabla u|)|\nabla u|^2- G(|\nabla u|)\right] d\sigma >0 .
$$
Hence, the Pohozaev's identity \eqref{Pohozaev indentity} and the $\Delta_2$-condition give
\begin{equation}
\begin{split}
0 & < (n+\alpha)\int_B |x|^\alpha H(u)\,dx + \int_B g(|\nabla u|)\dfrac{\nabla u}{|\nabla u|}\,dx -n\int_B G(|\nabla u|)\,dx \\ & = (n+\alpha)\int_B |x|^\alpha H(u)\,dx+ \int_B |x|^\alpha h(u)u\,dx-n\int_B G(|\nabla u|)\,dx\\ & \leq \left(\dfrac{n+\alpha}{q^{-}}+1\right)\int_B |x|^\alpha h(u)u\,dx-n\int_B G(|\nabla u|)\,dx \\ & \leq \left(\dfrac{n+\alpha}{q^{-}}+1\right)\int_B |x|^\alpha h(u)u\,dx-\dfrac{n}{p^{+}} \int_B g(|\nabla u|)|\nabla u|^2\,dx \\ & \leq \left(\dfrac{n+\alpha}{q^{-}}+1 -\dfrac{n}{p^{+}}\right)\int_B |x|^\alpha h(u)u\,dx.
\end{split}
\end{equation}Therefore, we arrive at a contradiction under assumption $q^-\geq (p^+)_\alpha^*$. 
\end{proof}
\appendix
\section{Mountain pass lemma}\label{sec.app}

In this section, we recall the Mountain Pass Lemma of Ambrosetti-Rabinowitz \cite{AR}. We start with the Palais-Smale condition.

\begin{definition} \label{def1}
We say that the functional $J$ satisfies the \emph{Palais-Smale compactness condition} if each sequence $\{u_k\}_{k\in\N}\subset X_{rad}(B)$ such that
\begin{itemize}
\item[(i)] $\{J(u_k)\}_{k\in\N}$ is bounded, and
\item[(ii)] $J'(u_k)\to 0$ in $X_{rad}(B)$
\end{itemize}
is precompact in $X_{rad}(B)$.
\end{definition}

We state the \emph{mountain-pass theorem} due to Ambrossetti  and Rabinowitz \cite{AR}.
\begin{theorem} \label{mountain}
Let $E$ be a Banach space and let $J\in C^1(E,\R)$ satisfy the Palais-Smale condition. Suppose that
\begin{itemize}
\item[(i)] $J(0)=0$ and $J(e)=0$ for some $e\neq 0$ in $E$;
\item[(ii)] there exists $\rho\in (0,\|e\|)$, $\sigma>0$ such that $J\geq \sigma$ in $S_\rho=\{u\in E \colon \|u\|=\rho\}$.
\end{itemize}
Then $J$ has a positive critical value
$$
c=\inf_{h\in \Gamma} \max_{t\in [0,1]} J(h(t)) \geq \sigma>0
$$
where $\Gamma=\{h\in C([0,1],E)\colon h(0)=0, h(1)=e\}$.
\end{theorem}

%

\end{document}